\documentclass[]{article}
\usepackage{amssymb,amsmath,amsthm}
\usepackage{color}

\newcommand{\R}{\mathbf{R}}

\newcommand {\E}{\mathrm{E}}

\renewcommand{\d}{\text{\rm d}}

\newcommand{\sG}{\mathcal{G}}

\newcommand{\sA}{\mathcal{A}}
\newcommand{\sE}{\mathcal{E}}

\newtheorem{stat}{Statement}[section]
\newtheorem{proposition}[stat]{Proposition}
\newtheorem{corollary}[stat]{Corollary}
\newtheorem{theorem}[stat]{Theorem}
\newtheorem{lemma}[stat]{Lemma}
\newtheorem{assumption}[stat]{Assumption}
\theoremstyle{definition}
\newtheorem{definition}[stat]{Definition}\newtheorem{remark}[stat]{Remark}

\numberwithin{equation}{section}

\numberwithin{equation}{section}

\newcommand{\RR}[1]{\mathbb{#1}}

\newcommand{\rd}{{\mathbb R^d}}

\def\R{{\mathbb R}}

\def\E{{\mathbb E}}

\allowdisplaybreaks

\begin{document}

\title{Non-linear noise excitation for some space-time fractional stochastic equations in bounded domains}

\author{Mohammud Foondun\\
Loughborough University\\
and \\
Jebessa B. Mijena\\
Georgia College and State University\\
 and \\
   Erkan Nane\\
 Auburn University}

\maketitle


\begin{abstract}
In this paper we study non-linear noise excitation for the following class of space-time fractional stochastic equations in  bounded domains:
$$\partial^\beta_tu_t(x)=-\nu(-\Delta)^{\alpha/2} u_t(x)+I^{1-\beta}_t[\lambda \sigma(u)\stackrel{\cdot}{F}(t,x)]$$ in $(d+1)$ dimensions, where $\nu>0, \beta\in (0,1)$, $\alpha\in (0,2]$. The operator  $\partial^\beta_t$ is the Caputo fractional derivative,  $-(-\Delta)^{\alpha/2} $ is the generator of an isotropic stable process and $I^{1-\beta}_t$ is the fractional integral operator. The forcing noise denoted by $\stackrel{\cdot}{F}(t,x)$ is a  Gaussian noise. The multiplicative non-linearity $\sigma:\RR{R}\to\RR{R}$ is assumed to be globally Lipschitz continuous. These equations  were recently  introduced by Mijena and Nane  \cite{nane-mijena-2014}.  We first study  the existence and uniqueness of the solution of these equations {and} under suitable conditions on the initial function, we {also} study the asymptotic behavior of the solution with respect to the parameter $\lambda$. In particular, our results are significant extensions of those  in \cite{foondun-tian-liu-2015}, \cite{foondun-khoshnevisan-09}, \cite{nane-mijena-2014}, and \cite{mijena-nane-2}.
\end{abstract}

Keywords: Space-time-fractional stochastic partial differential equations; space-time fractional diffusion in bounded domain;  fractional Duhamel's principle; Caputo derivatives; noise excitability.

\section{Introduction and  statement of the main results.}

{While fractional calculus has existed in the theoretical realm of mathematics as long as its classical counterpart, the pragmatic applications of said branch of calculus were sparse until the last century, when a rather large number of scientific branches, such as statistical mechanics, theoretical physics,  theoretical neuroscience, theory of complex chemical reactions, fluid dynamics, hydrology, and mathematical finance began applying fractional differential equations to problems in said fields; see, for example, Khoshnevisan \cite{khoshnevisan-cbms} for an extensive list of references. One such application is, the fractional heat equation $\partial^\beta_tu_t(x)=\Delta u_t(x)$, which describes heat propagation in inhomogeneous media whereas the integer counterpart, the classical heat equation $\partial_tu_t(x)=\Delta u_t(x)$, is used for modeling heat diffusion in homogeneous media. It is well known that when $0<\beta<1$,  time fractional equations are  known to exhibit sub diffusive behavior and are related with anomalous diffusions, or diffusions in non-homogeneous media, with random fractal structures; see, for instance, \cite{meerschaert-nane-xiao}. Most of the work done so far} on the stochastic heat equations have dealt with the usual time derivative, that is $\beta=1$. But recently, Mijena and Nane have introduced time fractional SPDEs in \cite{nane-mijena-2014}. These types of time fractional stochastic  equations are  attractive models that can be used to model phenomenon with random effects with thermal memory.  {They also studied exponential growth of solutions of time fractional SPDEs--intermittency-- under the assumption that the initial function is bounded from below in \cite{mijena-nane-2}}. In the paper \cite{FooNane} Foondun and Nane have proved asymptotics of the second moment of the solution under various assumptions on the initial function.
{In addition,} a related  class of  time-fractional  SPDE was studied by  Karczewska \cite{karczewska}, Chen  et al. \cite{chen-kim-kim-2014}, and Baeumer et al \cite{baeumer-Geissert-Kovacs}. { In these papers they} proved regularity of the solutions to the time-fractional parabolic type SPDEs using cylindrical Brownain motion in Banach spaces  in the sense of  \cite{daPrato-Zabczyk}. For a comparison of the two approaches to SPDE's see the paper by Dalang and Quer-Sardanyons \cite{Dalang-Quer-Sardanyons}.
 {The current paper  is mainly  about} a class of space-time fractional  stochastic  heat equations in bounded domains with Dirichlet boundary conditions.

{ Next we provide some heuristics before describing our equations. Fix $R>0$. First let us look at }the following space-time  fractional equation with Dirichlet boundary conditions (see  \cite{cmn-12} and  \cite{mnv-09}  for a representation of the solution),
\begin{equation}\label{tfpde-dirichlet}
\begin{aligned}
\partial^\beta_tu_t(x)&=-\nu(-\Delta)^{\alpha/2} u_t(x), \ \  x\in B(0,R), t>0,\\
u_t(x)&=0,\ \ \ x\in B(0,R)^c,
\end{aligned}
\end{equation}
with $\beta\in (0,1)$  and $\partial^\beta_t$ is the Caputo fractional derivative which first appeared in \cite{Caputo} and is defined by
\begin{equation}\label{CaputoDef}
\partial^\beta_t u_t(x)=\frac{1}{\Gamma(1-\beta)}\int_0^t \partial_r
u_r(x)\frac{\d r}{(t-r)^\beta} .
\end{equation}
If $u_0(x)$ denotes the initial condition to the above equation, then the solution can be written as
\begin{equation*}
u_t(x)=\int_{B(0,R)} G_B(t,x,y)u_0(y)\d y.
\end{equation*}
{$G_B(t,x,y)$} is the space-time fractional heat kernel. { Now consider}
 \begin{equation}\label{tfpde}
 \partial^\beta_tu_t(x)=-\nu(-\Delta)^{\alpha/2}u_t(x)+f(t,x),
 \end{equation}
with the same initial condition $u_0(x)$ and $f(t,x)$ is some nice function. {To get the correct version of \eqref{tfpde} we will make use of {\bf time fractional Duhamel's principle} \cite{umarov-06, umarov-12,Umarov-saydamatov}.
Applying} the fractional Duhamel principle, the solution to \eqref{tfpde} is given by
$$u_t(x)=\int_{B(0,R)} G_B({t},x,y)u_0(y)\d y+\int_0^t\int_{B(0,R)} G_B({t-r},x,y)\partial^{1-\beta}_r f(r,y)\d y\d r.$$

{Using the definition of the fractional integral
$$I^{\gamma}_tf(t):=\frac{1}{\Gamma(\gamma)} \int _0^t(t-\tau)^{\gamma-1}f(\tau)\d\tau,$$
and the property
$$ \partial _t^\beta I^\beta_t g(t)=g(t),$$
for every  $\beta\in (0,1)$, and   $g\in L^\infty(\R_+)$ or $g\in C(\R_+)$, then by the Duhamel's principle, the  mild solution to
\eqref{tfpde} where the force is $f(t,x)=I^{1-\beta}_tg(t,x)$, will be given by}
$$u_t(x)=\int_{B(0,R)} G_B(t,x,y)u_0(y)\d y+\int_0^t\int_{B(0,R)} G_B({t-r},x,y) g(r,y)\d y\d r.$$
{ For more informations on these see \cite{chen-kim-kim-2014} and \cite{nane-mijena-2014}}.

  {Our first equation in this paper is the following.}
\begin{equation}\label{tfspde}
\begin{split}
\partial^\beta_t u_t(x)&=-\mathcal{L} u_t(x)+I^{1-\beta}_t[\lambda\sigma(u_t(x))\stackrel{\cdot}{W}(t,x)],\, x\in B(0,R),\\
u_t(x)&=0 \ \ \ x\in B(0,R)^c,\\
 \end{split}
 \end{equation}
where $\mathcal{L}$ is the generator of an $\alpha$-stable process killed upon exiting $B(0,R)$, the initial datum $u_0$ is a non-random  nonnegative measurable function $u_0: B(0,R)\to\R_+$ which is strictly positive in a set of positive measure in $B(0,R)$. $\stackrel{\cdot}{W}(t,x)$ is a  space-time white noise with $x\in \R^d$ and $\sigma:\R\to\R$ is a globally Lipschitz function with $\sigma(0)=0$. $\lambda$ is a positive parameter called the ``level of noise". 

{We will use an idea in Walsh \cite{walsh} to make sense of the above equation.
 Using the above argument,} a solution $u_t$ to the above equation will in fact be a solution to the following integral equation.
\begin{equation}\label{mild-sol-white}
u_t(x)=(\mathcal{G}_B u_0)_t(x)+\lambda\int_0^t\int_{B(0,R)}G_B(t-s, x,y)\sigma(u_s(y))W(\d y\,\d s),
\end{equation}
where
$$
(\mathcal{G}_Bu_0)_t(x):=\int_{B(0,R)} G_B(t,x,y)u_0(y)\d y.
$$
Here $G_B(t,x,y)$ denotes the heat kernel of the space-time fractional diffusion equation with Dirichlet boundary conditions in \eqref{tfpde-dirichlet}. {In this paper $\alpha$ and $\beta$ are fixed}. We will restrict $\beta\in (0,\,1)$ and $\alpha\in (0,2]$. { The relation between the dimension $d$ and the parameters $\alpha$ and $\beta$ is given by}
\begin{equation*}
d<(2\wedge \beta^{-1})\alpha.
\end{equation*}
{Observe} that when $\beta=1$, the equation reduces to the well known stochastic heat equation and the above inequality restrict the problem to a one-dimensional one. This is the so called curse of dimensionality explored in \cite{foondun-khoshnevisan-Nualart-11}. We will need $d< 2\alpha$ to get a finite $L^2-$norm of the heat kernel, while $d<\beta^{-1}\alpha$ is needed for an integrability condition needed for ensuring existence and uniqueness of the solution. { We will require the following notion of "random-field" solution.}
\begin{definition}
A random field $\{u_t(x), \ t\geq 0, x\in  B(0,R)\}$ is called a mild solution of \eqref{tfspde} if
\begin{enumerate}
\item $u_t(x)$ is jointly measurable in $t\geq 0$ and $x\in  B(0,R)$;
\item $\forall (t,x)\in [0,\infty)\times \rd$, $\int_0^t\int_{\R^d}G_B(t-s, x,y)\sigma(u_s(y))W(\d y\,\d s)$ is well-defined in $L^2(\Omega)$; by the Walsh-Dalang isometry this is the same as requiring $$
    \sup_{x\in \rd}\sup_{0<t\leq T}\E|u_t(x)|^2<\infty\quad\text{for all}\quad T<\infty.
    $$
\item The following holds in $L^2(\Omega)$,
$$
u_t(x)=(\mathcal{G}_Bu_0)_t(x)+\lambda\int_0^t\int_{\R^d}G_B({t-s},x,y)\sigma(u_s(y))W(\d y\,\d s).
$$
\end{enumerate}
\end{definition}

{Before stating our main results, we will mention all the assumptions we need}.  The first assumption is required for the existence-uniqueness result as well as the upper bound on the second moment of the solution.

\begin{assumption}\label{existence-uniqueness}
\begin{itemize}
\item We assume that initial condition is a non-random bounded non-negative function $u_0:\R^d\rightarrow \R$.
\item We assume that $\sigma:\R\rightarrow \R$ is a globally Lipschitz function satisfying $\sigma(x)\leq L_\sigma|x|$ with $L_\sigma$ being a positive number.
\end{itemize}
\end{assumption}

The following assumption is needed for lower bound on the second moment.

\begin{assumption}\label{lowerbound}
\begin{itemize}
\item We will assume that the initial function $u_0$ is non-negative on a set of positive measure.
\item The function $\sigma$ satisfies $\sigma(x)\geq l_\sigma|x|$ with $l_\sigma$ being a positive number.
\end{itemize}
\end{assumption}

Our first  theorem extends the result of Mijena and Nane  \cite[Theorem 2]{nane-mijena-2014} for the equation \eqref{tfspde} in $\rd$ to  the equation with Dirichlet  boundary conditions in bounded domains.
\begin{theorem}\label{white:upperbound}
Suppose that $d<(2\wedge \beta^{-1})\alpha$.  Then under Assumption \ref{existence-uniqueness}, there exists a unique random-field solution to \eqref{tfspde} satisfying
\begin{equation*}
\sup_{x\in B(0,R)}\E|u_t(x)|^2\leq c_1e^{c_2\lambda^{\frac{2\alpha}{\alpha-d\beta}}t}\quad \text{for\,all}\quad t>0.
\end{equation*}
Here $c_1$ and $c_2$ are positive constants.
\end{theorem}
\begin{remark}
This theorem says that second moment grows at most exponentially. For small $\lambda$, this upper bound is not sharp. But for large $\lambda$, this is sharp.
 Theorem \ref{white:upperbound} implies that a random field solution exists when  $d<(2\wedge \beta^{-1})\alpha$.  It follows from this theorem that space-time fractional stochastic equations with  space-time white noise  is that a random field solution exists in space dimension greater than 1 in some cases, in contrast to the  parabolic stochastic heat type equations, the case $\beta=1$. So in the case $\alpha=2, \beta<1/2$, a random field solution exists when $d=1,2,3$.
 When $\beta=1$ a random field solution exist only in spatial dimension $d=1$.

\end{remark}

\begin{remark}

Suppose that $d<(2\wedge \beta^{-1})\alpha$.
Using similar ideas in the proof of Theorem \ref{white:upperbound}, the results in Theroem \ref{white:upperbound} can be extended to other classes of bounded domains in $\rd$. Let $D$ be a bounded domain that is regular as in \cite{cmn-12}. Consider  the equation
\begin{equation}\label{tfspde-bounded-domain}
\begin{split}
\partial^\beta_t u_t(x)&=-\mathcal{L} u_t(x)+I^{1-\beta}_t[\lambda\sigma(u_t(x))\stackrel{\cdot}{W}(t,x)],\,t>0,  x\in D,\\
u_t(x)&=0 \ \ \ x\in D^C,\\
 \end{split}
 \end{equation}
where $\mathcal{L}$ is the generator of an $\alpha$-stable process killed upon exiting $D$, the initial datum $u_0$ is a non-random  nonnegative measurable function $u_0: D\to\R_+$ which is strictly positive in a set of positive measure in $D$. $\stackrel{\cdot}{W}(t,x)$ is a  space-time white noise with $x\in \R^d$,  and $\sigma:\R\to\R$ is a globally Lipschitz function with $\sigma(0)=0$.
Then under Assumption \ref{existence-uniqueness}, there exists a unique random-field solution to \eqref{tfspde-bounded-domain} satisfying
\begin{equation*}
\sup_{x\in D}\E|u_t(x)|^2\leq c_1e^{c_2\lambda^{\frac{2\alpha}{\alpha-d\beta}}t}\quad \text{for\,all}\quad t>0.
\end{equation*}
Here $c_1$ and $c_2$ are positive constants.

\end{remark}

Our second theorem is the following.
\begin{theorem}\label{main-thm-white-noise}
Fix $\epsilon>0$ and let $x\in B(0, R-\epsilon)$, then for any $t>0$,
\begin{equation*}
\lim_{\lambda\rightarrow \infty} \frac{\log \log \E|u_t(x)|^2}{\log \lambda}=\frac{2\alpha}{\alpha-d\beta},
\end{equation*}
where $u_t$ is the mild solution to \eqref{tfspde}.
\end{theorem}

Set
\begin{equation*}
\sE_t(\lambda):=\sqrt{\int_{\R^d}\E|u_t(x)|^2\,\d x}.
\end{equation*}

and define the nonlinear excitation index by
\begin{equation*}
e(t):=\lim_{\lambda \rightarrow \infty }\frac{\log \log \sE_t(\lambda)}{\log \lambda}.
\end{equation*}

\begin{corollary}\label{corollary-excitation}
The excitation index of the solution to \eqref{tfspde}, $e(t)$ is  equal to $\frac{2\alpha}{\alpha-d\beta}$.
\end{corollary}

\noindent { The  second class of equation we introduce in this paper is with space colored noise stated as:}
\begin{equation}\label{tfspde-colored}
\begin{split}
\partial^\beta_t u_t(x)&=\mathcal{L} u_t(x)+I_t^{1-\beta}[\lambda \sigma(u_t(x))\dot{F}(t,\,x)],\, x\in B(0,R), t>0,\\
u_t(x)&=0,\ \ \ x\in B(0,R)^c.
\end{split}
\end{equation}

The only difference with \eqref{tfspde} is that the noise term is now colored in space. All the other conditions are the same. We now briefly describe the noise.
$\dot{F}$ denotes the Gaussian colored noise satisfying the following property,
$$
\E[\dot{F}(t,x)\dot{F}(s,y)]=\delta_0(t-s)f(x,y).
$$
{This can be interpreted more formally as}

\begin{equation}\label{covariance-colored}
Cov \bigg(\int \phi \d F, \int \psi \d F\bigg)=\int_0^\infty\int _{\rd}\d x\int_{\rd}\d y\phi_s(x)\psi_s(y)f(x-y),
\end{equation}
where we use the notation $\int \phi \d F$ to denote the wiener integral of $\phi$ with respect to $F$, and  the right-most integral converges absolutely.

We will assume that the spatial correlation of the noise term is given by the following function for  $\gamma<d$,
$$
f(x,y):=\frac{1}{|x-y|^\gamma}.
$$

Following Walsh \cite{walsh}, we define the mild solution of \eqref{tfspde-colored}  as the predictable solution to the following integral equation

\begin{equation}\label{mild-sol-colored}
\begin{split}
u_t(x)&=(\mathcal{G}_Bu_0)_t(x)+\lambda \int_{B(0,R)} \int_0^t G_B({t-s},x,y)\sigma(u_s(y))F(\d s \d y).
\end{split}
\end{equation}
As before, we will look at random field solution, which is defined by \eqref{mild-sol-colored}. We will also assume the following
\begin{equation*}
\gamma <\alpha\wedge d.
\end{equation*}
{ The condition we should have $\gamma<d$ follows from an integrability condition about the correlation function and $\gamma<\alpha$ }comes from an integrability condition needed for the existence and uniqueness of the solution.
Our first result on space colored noise case reads as follows.
\begin{theorem}\label{thm-colored-noise} Under the Assumption \ref{existence-uniqueness},
there exists a unique random field solution $u_t$  of \eqref{tfspde-colored} whose second moment satisfies
\begin{equation*}
\sup_{x\in B(0,R)}\E|u_t(x)|^2\leq c_5 \exp(c_6\lambda^{2\alpha/(\alpha-\gamma\beta)}t)\quad\text{for\,all}\quad t>0.
\end{equation*}
Here the constants $c_5,  c_6$  are positive numbers.
 \end{theorem}

Our main result for the space colored noise equation is the following theorem.

\begin{theorem}\label{excitation-colored}
Fix $t>0$ and $x\in B(0, R-\epsilon)$, we then have
\begin{equation*}
\lim_{\lambda\rightarrow \infty} \frac{\log \log \E|u_t(x)|^2}{\log \lambda}=\frac{2\alpha}{\alpha-\gamma\beta},
\end{equation*}
where $u_t$ is the unique solution to \eqref{tfspde-colored}.
\end{theorem}

\begin{corollary}\label{corollary-excitation-colored}
The excitation index of the solution to \eqref{tfspde-colored}, $e(t)$ is  equal to $\frac{2\alpha}{\alpha-\gamma\beta}$.
\end{corollary}

A key difference from the methods used in \cite{foondun-tian-liu-2015} and \cite{foodun-liu-omaba-2014} is that, here we need to overcome some new technical difficulties.  Compared with the usual  heat equation with the same boundary conditions, time fractional equations have significantly different behavior. This is the source of the main difficulties we have to overcome. Our method will rely on heat kernel estimates which we will prove later on.

{We now briefly give an outline of the paper. In this paper we employ similar  methods as  in \cite{FooNane} and \cite{foondun-tian-liu-2015} with crucial changes to prove our main results.  We give some preliminary results in section 2. We prove a number of interesting properties of the heat kernel of the time fractional heat type partial differential equations that are essential to the proof of our main results. The most important result in this section is Lemma \ref{density-lower-bound-killed-fractional}. The proofs of the results in the space-time white noise are given in Section 3. In Section 4, we prove the main  results about the space colored noise equation. We give and extension of the results stated in the introduction in section 5.
Throughout the paper, we use the letter $C$ or $c$ with or without subscripts to denote
a constant whose value is not important and may vary from places to places.  If $x\in \R^d$, then $|x|$ will denote the euclidean norm of $x\in\R^d$, while when $A\subset \R^d$, $|A|$ will denote the Lebesgue measure of $A$.}


\section{Preliminaries.}
As mentioned in the introduction, the behaviour of the heat kernel $G_B(t, x)$ will play an important role. This section will mainly be devoted to estimates involving this quantity.   Let $X_t$ denote a symmetric $\alpha$ stable process with density function denoted by $p(t,\,x)$. This is characterized through the Fourier transform which is given by

\begin{equation}\label{Eq:F_pX}
\widehat{p(t,\,\xi)}=e^{-t\nu|\xi|^\alpha}.
\end{equation}

Let $D=\{D_r,\,r\ge0\}$ denote a $\beta$-stable subordinator and $E_t$ be its first passage time. It is known that the density of the time changed process $X_{E_t}$ is given by the $G_t(x)$. By conditioning, we have

\begin{equation}\label{Eq:Green1}
G_t(x)=\int_{0}^\infty p(s,\,x) f_{E_t}(s)\d s,
\end{equation}
where
\begin{equation}\label{Etdens0}
f_{E_t}(x)=t\beta^{-1}x^{-1-1/\beta}g_\beta(tx^{-1/\beta}),
\end{equation}
where $g_\beta(\cdot)$ is the density function of $D_1$ and is infinitely differentiable on the entire real line, with $g_\beta(u)=0$ for $u\le 0$. Moreover,
\begin{equation}\label{Eq:gbeta0}
g_\beta(u)\sim K(\beta/u)^{(1-\beta/2)/(1-\beta)}\exp\{-|1-\beta|(u/\beta)^{\beta/(\beta-1)}\}\quad\mbox{as}\,\, u\to0+,
\end{equation}
and
\begin{equation}\label{Eq:gbetainf}
g_\beta(u)\sim\frac{\beta}{\Gamma(1-\beta)}u^{-\beta-1} \quad\mbox{as}\,\, u\to\infty.
\end{equation}

 We will need  the following properties of the heat kernel of stable process.
  \begin{itemize}
\item \begin{equation*}
p(t,\,x)=t^{-d/\alpha}p(1,\,t^{-1/\alpha}x).
\end{equation*}
\item \begin{equation*}
p(st,\,x)=s^{-d/\alpha}p(t,\,s^{-1/\alpha}x).
\end{equation*}
\item $p(t,\,x)\geq p(t,\,y)$ whenever $|x|\leq |y|$.
\item It is well known that the transition density $p(t,\,x)$ of any strictly stable process satisfies the following
 \begin{equation}\label{stable-density-bounds}
 c_1\bigg(t^{-d/\alpha}\wedge \frac{t}{|x|^{d+\alpha}}\bigg)\leq p(t,\,x)\leq c_2\bigg(t^{-d/\alpha}\wedge \frac{t}{|x|^{d+\alpha}}\bigg),
 \end{equation}
 where $c_1$ and $c_2$ are positive constants.

\end{itemize}



The $L^2$-norm of the heat kernel can be calculated as follows.
\begin{lemma} [ Lemma 1 in \cite{nane-mijena-2014}]\label{Lem:Green1} Suppose that $d < 2\alpha$, then
\begin{equation}\label{Eq:Greenint}
\int_{{\R^d}}G^2_t(x)\d x  =C^\ast t^{-\beta d/\alpha},
\end{equation}
where the constant $C^{\ast}$ is given by
\begin{equation*}
C^\ast = \frac{(\nu )^{-d/\alpha}2\pi^{d/2}}{\alpha\Gamma(\frac d2)}\frac{1}{(2\pi)^d}\int_0^\infty z^{d/\alpha-1} (E_\beta(-z))^2 \d z.
\end{equation*}
Here $E_\beta(x)$ is the Mittag-Leffler function defined by
\begin{equation}\label{ML-function}
E_\beta(x) = \sum_{k=0}^\infty\frac{x^k}{\Gamma(1+\beta k)}.
\end{equation}
\end{lemma}

\begin{remark}Let $D\subset \rd$ be a bounded domain.
Let $p_D(t,x,y)$ denote the heat kernel of the equation \eqref{tfpde-dirichlet} when $\beta=1$. A well known fact is that
\begin{equation}\label{upper-bound-killed-stable}
p_D(t,x,y)\leq p(s,x,y)\ \ \mathrm{for \ all}\  x, y\ \in D, t>0.
\end{equation}

 Using the representation from Meerschaert et al.  \cite{cmn-12} and \cite{mnv-09}
$$
G_D(t,x,y)=\int_0^\infty p_D(s,x,y)f_{E_t}(s)ds.
$$
and equation \eqref{upper-bound-killed-stable}
we get
\begin{equation}\label{bounded-free-upper-bound}
G_D(t,x,y)\leq G(t,x,y)\ \ \mathrm{for \ all}\  x, y\ \in D, t>0.
\end{equation}
This fact will be crucial in proving the existence and uniqueness of solutions to equations \eqref{tfspde} and \eqref{tfspde-colored}.

\end{remark}

The next result is crucial in getting the upper bounds for the spatially colored noise equation.
 \begin{lemma}[Lemma 2.7 in \cite{FooNane}]\label{lemma:covariance-upper-bound}
Suppose that $\gamma<\alpha$, then there exists a constant $c_1$ such that for all $x,\,y\in \R^d$, we have
\begin{equation*}
\int_{\R^d}\int_{\R^d}G_t(x-w)G_t(y-z)f(z,\,w)\d w \d z\leq \frac{c_1}{t^{\gamma\beta/\alpha}}.
\end{equation*}
 \end{lemma}

The next lemma follows from the previous lemma by using \eqref{bounded-free-upper-bound}.
\begin{lemma}\label{lemma:covariance-upper-bound-b-domain} Let $D\subset\rd$ be a bounded domain.
Suppose that $\gamma<\alpha$, then there exists a constant $c_1$ such that for all $x,\,y\in D$, we have
\begin{equation*}
\int_{D}\int_{D}G_D(t, x,w)G_D(t, y,z)f(z,\,w)\d w \d z\leq \frac{c_1}{t^{\gamma\beta/\alpha}}.
\end{equation*}
 \end{lemma}
The next proposition  is the crucial result in proving the lower bounds in Theorems \ref{white:upperbound} and \ref{excitation-colored}.

\begin{proposition}\label{density-lower-bound-killed-fractional}
Fix $\epsilon>0$, then there exists $t_0>0$ such that for all $x,y\in B(0, R-\epsilon)$ and for all $t<t_0$ and  $|x-y|<t^{\beta/\alpha}$ we have
$$
G_B(t,x,y)\geq C t^{-\beta d/\alpha},
$$
for some constant $C>0$.

\end{proposition}
\begin{proof}
We use the representation
$$
G_B(t,x,y)=\int_0^\infty p_B(s,x,y)f_{E_t}(s)ds.
$$

By proposition 2.1 in \cite{foondun-tian-liu-2015}  there exists a $T_0>0$ such that $p_D(t,x,y )\geq c_1 p(t,x,y )$ whenever $t\leq T_0$.
 Now  this and the representation \eqref{Etdens0} with a change of variables  we get
 \begin{eqnarray}
 G_B(t,x,y)&\geq & \int_0^{T_0} p_B(s,x,y)f_{E_t}(s)ds\nonumber\\
 &\geq &c_1\int_0^{T_0} p(s,x,y)f_{E_t}(s)ds\nonumber\\
 &=&c_1\int_{tT_0^{-1/\beta}}^\infty p((t/u)^\beta,x,y)g_\beta(u)du.
 \end{eqnarray}
  Now suppose that $tT_0^{-1/\beta}<1/2$ and $|x-y|<t^{\beta/\alpha}$ hence $t/|x-y|^{\alpha/\beta}>1$ and for $u<t/|x-y|^{\alpha/\beta}$ we have $(t/u)^\beta>|x-y|^\alpha$ or equivalently $|x-y|<[(t/u)^\beta]^{1/\alpha}$, therefore using all of these observations with \eqref{stable-density-bounds} we obtain
  \begin{eqnarray}
 G_B(t,x,y)&\geq & c_1\int_{tT_0^{-1/\beta}}^\infty p((t/u)^\beta,x,y)g_\beta(u)du\nonumber\\
 &\geq & c_1\int_{tT_0^{-1/\beta}}^{t/|x-y|^{\alpha/\beta}} p((t/u)^\beta,x,y)g_\beta(u)du\nonumber\\
 &\geq &C \int_{tT_0^{-1/\beta}}^{t/|x-y|^{\alpha/\beta}} (1/(t/u)^{\beta})^{d/\alpha}g_\beta(u)du\nonumber\\
 &\geq &C\int_{1/2}^{1} t^{-\beta d/\alpha} u^{\beta d/\alpha}g_\beta(u)du\nonumber\\
 &=&  C t^{-\beta d/\alpha}\int_{1/2}^{1} u^{\beta d/\alpha}g_\beta(u)du= C t^{-\beta d/\alpha}.\nonumber
 \end{eqnarray}
\end{proof}

\begin{remark}\label{lower-bound-frac-int-rep}
Recall that for any $t>0$ and $x\in B(0,R)$
$$
(\mathcal{G}_Bu)_t(x):=\int _{B(0,R)}G_B(t,x,y)u_0(y)dy.
$$
By remark 2.2 in Foondun et al. \cite{foondun-tian-liu-2015} we know that for fixed $\epsilon>0$ we  have
$$h_t:=\inf_{x\in B(0,R-\epsilon)}\inf_{s\leq t} {(\tilde{\mathcal{G}}_Bu)_s}(x)=\inf_{x\in B(0,R-\epsilon)} {(\tilde{\mathcal{G}}_Du)_t}(x)>0,$$
where ${(\tilde{\mathcal{G}}_Bu)_s}(x)=\int _{B(0,R)}p_B(t,x,y)u_0(y)dy$ is the killed semigroup of stable process and is the solution of \eqref{tfpde-dirichlet} when  $\beta=1$.

By a simple conditioning we have
 \begin{equation}
 \begin{split}
(\mathcal{G}_Bu)_{s+t_0}(x)&:=\int _{B(0,R)}G_B(s+t_0,x,y)u_0(y)dy=\int_0^\infty{(\tilde{\mathcal{G}}_B u)_{s'}}(x) f_{E_{s+t_0}}(s')ds'\\
&\geq \int_0^{s+t_0} {(\tilde{\mathcal{G}}_Bu)_s}(x) f_{E_{s+t_0}}(s)ds\\
&\geq h_{s+t_0} \int_0^{s+t_0} f_{E_{s+t_0}}(s)ds:=g_t>0,
\end{split}
\end{equation}
for any fixed $t_0>0$ and all $s\leq t$.
\end{remark}

We end this section with a few results from \cite{foondun-tian-liu-2015} and  \cite{foodun-liu-omaba-2014}. These will  be useful for the proofs of our main results.
\begin{proposition}\label{prop:renewal-upper}[Proposition 2.5 in \cite{foodun-liu-omaba-2014}]
Let $\rho>0$ and suppose $f(t)$ is a locally integrable function satisfying
$$
f(t)\leq c_1+\kappa \int_0^t(t-s)^{\rho-1} f(s)\d s \ \ \mathrm{for \ all} \ \ t>0,
$$
where $c_1$ is some positive number. Then, we have
$$
f(t)\leq c_2\exp(c_3(\Gamma(\rho))^{1/\rho}\kappa^{1/\rho} t)\  \ \mathrm{for \ all} \ \ t>0,
$$ for some positive constants $c_2$ and $c_3$.
\end{proposition}
Also we give the following converse.
\begin{proposition}[Proposition 2.6 in \cite{foodun-liu-omaba-2014}]
Let $\rho>0$ and suppose $f(t)$ is nonnegative,  locally integrable function satisfying
$$
f(t)\geq c_1+\kappa \int_0^t(t-s)^{\rho-1} f(s)\d s \ \ \mathrm{for \ all} \ \ t>0,
$$
where $c_1$ is some positive number. Then, we have
$$
f(t)\geq c_2\exp(c_3(\Gamma(\rho))^{1/\rho}\kappa^{1/\rho} t)\  \ \mathrm{for \ all} \ \ t>0,
$$ for some positive constants $c_2$ and $c_3$.
\end{proposition}

\begin{proposition}[Proposition 2.6 in \cite{foondun-tian-liu-2015}] \label{prop2.6-foondun}
Let $T< \infty$ and $\eta>0$. Suppose that $f(t)$ is a positive  locally integrable function satisfying
 \begin{equation}
 f(t)\geq c_2+ \kappa\int_0^t(t-s)^{\eta-1}f(s)ds\ \ \mathrm{for\ all} \ 0\leq t\leq T,
 \end{equation}
where $c_2$ is some positive number. Then for any $t\in (0, T]$, we have the following
$$
\liminf_{\kappa\to\infty}\frac{\log \log f(t)}{\log \lambda}\geq \frac1\eta.
$$
\end{proposition}

\begin{lemma}[Lemma 2.4 in \cite{foondun-tian-liu-2015}]\label{lower-bound-lambda}
Let $\rho>0$  and  $S(t)=\sum_{k=1}^\infty\left(\frac{t}{k^\rho}\right)^k$. For any fixed $t>0$, we have
$$
\liminf_{\theta\to\infty}\frac{\log \log  S(\theta t)}{\log \theta}\geq \frac1\rho.
$$

\end{lemma}

\section{Proofs for the white noise case.}

\subsection{Proofs of Theorem \ref{white:upperbound}.}
\begin{proof}
The proof follows main steps in  \cite{foondun-tian-liu-2015} and \cite{FooNane} with some crucial changes.
We first show the existence of a unique solution. This follows from a standard Picard iteration; see \cite{walsh}, so we just briefly spell out the main ideas. For more information, see \cite{nane-mijena-2014}. Set
\begin{equation*}
u_t^{(0)}(x):=(\sG_B u_0)_t(x)
\end{equation*}
and
\begin{equation*}
u_t^{(n+1)}(x):=(\sG_B u_0)_t(x)+\lambda\int_0^t\int_{B(0,R)}G_B({t-s},x,y)\sigma(u^{(n)}_s(y))W(\d y\,\d s)\quad\text{for}\quad n\geq 0.
\end{equation*}
Define $D_n(t\,,x):=\E|u^{(n+1)}_t(x)-u^{(n)}_t(x)|^2$ and $H_n(t):=\sup_{x\in \R^d}D_n(t\,,x)$. We will prove the result for $t\in [0,\,T]$, where $T$ is some fixed number. We now use this notation, \eqref{bounded-free-upper-bound},  together with Walsh's isometry and the assumption on $\sigma$ to write
\begin{equation*}
\begin{aligned}
D_n(t,\,x)&=\lambda^2\int_0^t\int_{B(0,R)}G^2_B({t-s}, x, y)\E|\sigma(u^{(n)}_s(y))-\sigma(u^{(n-1)}_s(y))|^2\d y\,\d s\\
&\leq \lambda^2L_\sigma^2\int _0^tH_{n-1}(s)\int_{\R^d}G^2({t-s},x,y)\,\d y\,\d s\\
&\leq \lambda^2L_\sigma^2\int_0^T\frac{H_{n-1}(s)}{(t-s)^{d\beta/\alpha}}\,\d s
\end{aligned}
\end{equation*}
We therefore have
\begin{equation*}
H_{n}(t)\leq \lambda^2L_\sigma^2\int_0^T\frac{H_{n-1}(s)}{(t-s)^{d\beta/\alpha}}\,\d s.
\end{equation*}
We now note that the integral appearing on the right hand side of the above display is finite when $d<\alpha/\beta$. Hence, by Lemma 3.3 in Walsh \cite{walsh}, the series $\sum_{n=0}^\infty H^{\frac12}_n(t)$ converges uniformly on $[0,\,T].$ Therefore, the sequence $\{u_n\}$ converges  in $L^2$ and uniformly on $[0,\,T]\times\R^d$ and the limit satisfies (\ref{mild-sol-white}).  We can prove uniqueness in a similar way.
We now turn to the proof of the exponential bound. From Walsh's isometry, we have
\begin{equation*}
\E|u_t(x)|^2=|(\sG_B u_0)_t(x)|^2+\lambda^2\int_0^t\int_{B(0,R)}G^2_B({t-s},x,y)\E|\sigma(u_s(y))|^2\d y\,\d s.
\end{equation*}
Since we are assuming that the initial  function (condition) is bounded, we have that $|(\sG_B u_0)_t(x)|^2\leq c_1$ and by \eqref{bounded-free-upper-bound} the second term is bounded by
\begin{equation*}
\begin{aligned}
\lambda^2L_\sigma^2\int_0^t\int_{B(0,R)}&G^2_B({t-s}, x, y)\E|u_s(y)|^2\d y\,\d s\\
&\leq c_1\lambda^2L_\sigma^2\int_0^t\frac{1}{(t-s)^{d\beta/\alpha}}\sup_{y\in B(0,R)}\E|u_s(y)|^2\d y\,\d s.
\end{aligned}
\end{equation*}
We therefore have
\begin{equation*}
\sup_{x\in B(0,R)}\E|u_s(x)|^2\leq c_1+c_2\lambda^2L_\sigma^2\int_0^t\frac{1}{(t-s)^{d\beta/\alpha}}\sup_{y\in B(0,R)}\E|u_s(y)|^2\,\d s.
\end{equation*}
The renewal inequality in Proposition \ref{prop:renewal-upper} with $\rho=(\alpha-d\beta)/\alpha$  proves the result.
\end{proof}

\subsection{Proof of Theorem \ref{main-thm-white-noise}.}
Set $\mathcal{S}_t(\lambda):=\sup_{x\in B(0,R)}\E  |u_t(x)|^2$.
We first state following proposition which follows from Theorem \ref{white:upperbound}.

\begin{proposition}
Fix $t>0$, then
$$
\limsup_{\lambda\to\infty}\frac{\log \log \mathcal{S}_t(\lambda)}{\log \lambda}\leq \frac{2\alpha}{\alpha-\beta d}.
$$
\end{proposition}
 For any fixed $\epsilon>0$, set
 $$
 \mathcal{I}_{\epsilon, t}(\lambda):=\inf_{x\in B(0,R-\epsilon)} \E|u_t(x)|^2.
 $$
 Next we give a proposition that gives the lower bound in Theorem \ref{main-thm-white-noise}
\begin{proposition}
 For any fixed $\epsilon>0$, there exists a $t_0>0$ such that for all $t\leq t_0$,
$$
\liminf_{\lambda\to\infty}\frac{\log \log \mathcal{I}_{\epsilon, t}(\lambda)}{\log \lambda}\geq \frac{2\alpha}{\alpha-\beta d}.
$$
\end{proposition}

\begin{proof}
The proof of  the proposition  will rely on the following observation. From Walsh isometry, we have
\begin{equation*}
\begin{split}
\E|u_t(x)|^2&=|(\sG_B u_0)_t(x)|^2+\lambda^2\int_0^t\int_{B(0,R)}G^2_B({t-s}, x,y)\E|\sigma(u_s(y))|^2\d y\,\d s.\\
&=I_1+I_2.
\end{split}
\end{equation*}
 We fix $\epsilon >0$ and choose a $t_0$  as in Proposition  \ref{density-lower-bound-killed-fractional}.

 For  $x\in B(0,R-\epsilon)$ we have $\mathcal{G}_B(t,x)\geq g_{t_0}$ by Remark \ref{lower-bound-frac-int-rep}.
 Hence $I_1\geq g_{t_0}^2.$ We now prove the lower bound for $I_2$.
 \begin{equation*}
 \begin{split}
 I_2&\geq (\lambda l_\sigma)^2 \int_0^t\int_{B(0,R)}G^2_B({t-s}, x,y)\E|u_s(y))|^2\d y\,\d s\\
 &\geq (\lambda l_\sigma)^2 \int_0^t \mathcal{I}_{\epsilon, s}(\lambda) \int_{B(0,R-\epsilon)}G^2_B({t-s}, x,y)\d y\,\d s.
 \end{split}
 \end{equation*}

 Set $A:=\{y\in B(0, R-\epsilon): \ |x-y|\leq (t-s)^{\beta/\alpha}\}$. Since $t-s\leq t_0$, we have
 $|A|\geq c_1 (t-s)^{d\beta/\alpha}$. Now using Proposition \ref{density-lower-bound-killed-fractional}, we have
 \begin{equation*}
 \begin{split}
 \int_{B(0,R-\epsilon)}G^2_B({t-s}, x,y)\d y&\geq c_2\int_A\frac{1}{(t-s)^{2\beta d/\alpha}}\\
 &=c_3\frac{1}{(t-s)^{\beta d/\alpha}}.
 \end{split}
 \end{equation*}
 We thus have
 $$
 I_2\geq  c_4 \lambda^2 \int_{0}^t  \frac{\mathcal{I}_{\epsilon, s}(\lambda) }{(t-s)^{\beta d/\alpha}}ds.
  $$

  Combining the above estimates we have
  $$
  \mathcal{I}_{\epsilon, t}(\lambda)\geq g_{t_0}^2+c_4 \lambda^2 \int_{0}^t  \frac{\mathcal{I}_{\epsilon, s}(\lambda) }{(t-s)^{\beta d/\alpha}}ds.
  $$

  We now apply Proposition \ref{prop2.6-foondun}.

 \end{proof}

 \begin{proof}[Proof of Theorem \ref{main-thm-white-noise}]

 The proof of the result when $t\leq t_0$ follows from the two propositions above. To prove the theorem for all $t>0$, we only need to prove the above proposition for all $t>0$. For any fixed $T, t>0$, by  changing variables we have
 \begin{equation*}
 \begin{split}
\E|u_{t+T}(x)|^2&\geq |(\sG_B u_0)_{t+T}(x)|^2+\lambda^2 \int_0^{t+T}\int_{B(0,R)}G^2_B({t+T-s},x,y)\E|\sigma(u_{s}(y))|^2\d y\,\d s\\
&\geq |(\sG_B u_0)_{t+T}(x)|^2+\lambda^2 \int_0^T\int_{B(0,R)}G^2_B({t+T-s},x,y)\E|\sigma(u_{s}(y))|^2\d y\,\d s\\
&+\lambda^2\int_0^t\int_{B(0,R)}G^2_B({t+T-s},x,y)\E|\sigma(u_{s+T}(y))|^2\d y\,\d s.
\end{split}
\end{equation*}
This gives
$$
\E|u_{t+T}(x)|^2\geq |(\sG_B u_0)_{t+T}(x)|^2+\lambda^2l_\sigma^2\int_0^t\int_{B(0,R)}G^2_B({t+T-s},x,y)\E|u_{s+T}(y)|^2\d y\,\d s,
$$
 since $|(\sG_B u_0)_{t+T}(x)|^2$ strictly positive, we can use the proof of the above proposition  with an obvious modification to conclude that
 $$
 \liminf_{\lambda\to\infty} \frac{\log\log \E|u_{t+T}(x)|^2}{\log \lambda}\geq \frac{2\alpha}{\alpha-\beta d},
 $$
  for $x\in B(0, R-\epsilon)$ and small $t$.
 \end{proof}

 \begin{proof}[Proof of Corollary \ref{corollary-excitation}]
 Note that
 $$
 \int_{B(0, R)}\E|u_t(x)|^2dx\leq C R^d\sup_{x\in B(0,R)}\E|u_t(x)|^2,
 $$
 and
 $$
  \int_{B(0, R)}\E|u_t(x)|^2dx\geq C (R-\epsilon)^d\inf_{x\in B(0,R-\epsilon)}\E|u_t(x)|^2.
 $$
 We now apply Theorem \ref{main-thm-white-noise} and use the definition of $\mathcal{E}_t(\lambda)$ to obtain the result.
 \end{proof}

\section{Proofs for the colored noise case.}
\subsection{ Proof of Theorem \ref{thm-colored-noise}.}
\begin{proof}
The proof of existence and uniqueness is standard as in  \cite{foondun-tian-liu-2015} and \cite{FooNane}. We give the details for the convenience of the reader.  For more information, see \cite{walsh}.  We set
\begin{equation*}
u^{(0)}(t,\,x):=(\sG_B u_0)_t(x),
\end{equation*}
and
\begin{equation*}
u^{(n+1)}(t,\,x):=(\sG_B u_0)_t(x)+\lambda\int_0^t\int_{B(0,R)}G_B({t-s}, x, y)\sigma(u^{(n)}(s,\,y))F(\d y\,\d s), \quad n\geq 0.
\end{equation*}
Define $D_n(t\,,x):=\E|u^{(n+1)}(t,\,x)-u^{(n)}(t,\,x)|^2$,  $H_n(t):=\sup_{x\in \R^d}D_n(t\,,x)$ and $\Sigma(t,y,n)=|\sigma(u^{(n)}(t,\,y))-\sigma(u^{(n-1)}(t,\,y))|$. We will prove the result for $t\in [0,\,T]$ where $T$ is some fixed number. We now use this notation together with the covariance formula \eqref{covariance-colored} and the assumption on $\sigma$ to write
\begin{equation*}
\begin{aligned}
&D_n(t,\,x)  \\
&=\lambda^2\int_0^t\int_{B(0,R)}\int_{B(0,R)} G_B({t-s}, x,y)G_B({t-s}, x, z)\E[\Sigma(s,y,n) \Sigma(s,z,n)] f(y,z)\d y d z \d s.
\end{aligned}
\end{equation*}
Now we estimate the expectation on the right hand side using Cauchy-Schwartz inequality.
\begin{equation*}
\begin{aligned}
\E[\Sigma(s,y,n) \Sigma(s,z,n)] &\leq L_\sigma^2\E|u^{(n)}(s,\,y)-u^{(n-1)}(s,\,y)||u^{(n)}(s,\,z)-u^{(n-1)}(s,\,z)|\\
&\leq L_\sigma^2\bigg(\E|u^{(n)}(s,\,y)-u^{(n-1)}(s,\,y)|^2\bigg)^{1/2}\\
&\ \ \ \ \bigg(\E|u^{(n)}(s,\,z)-u^{(n-1)}(s,\,z)| ^2\bigg)^{1/2}\\
&\leq L_\sigma^2 \bigg( D_{n-1}(s,y) D_{n-1}(s,z)\bigg)^{1/2}\\
&\leq L_\sigma^2 H_{n-1}(s).
\end{aligned}
\end{equation*}
Hence we have for $\gamma<\alpha$ using  Lemma  \ref{lemma:covariance-upper-bound}
\begin{equation*}
\begin{aligned}
&D_n(t,\,x)  \\
&\leq \lambda^2L_\sigma^2\int_0^t H_{n-1}(s)\int_{B(0,R)}\int_{B(0,R)} G_B({t-s}, x,y)G_B(,x , z) f(y,z)\d y \d z \,\d s\\
&\leq \lambda^2L_\sigma^2\int_0^t H_{n-1}(s)\int_{\R^d}\int_\rd G_{t-s}(x-y)G_{t-s}(x-z) f(y,z)\d y \d z \,\d s\\
&\leq c_1 \lambda^2L_\sigma^2\int _0^t \frac{H_{n-1}(s)}{(t-s)^{\gamma \beta/\alpha}}\,\d s.
\end{aligned}
\end{equation*}
We therefore have
\begin{equation*}
H_{n}(t)\leq c_1 \lambda^2L_\sigma^2\int _0^t \frac{H_{n-1}(s)}{(t-s)^{\gamma \beta/\alpha}}\,\d s.
\end{equation*}
We now note that the integral appearing on the right hand side of the above display is finite when $\gamma<\alpha/\beta$. Hence, by Lemma 3.3 in Walsh \cite{walsh}, the series $\sum_{n=0}^\infty H^{\frac12}_n(t)$ converges uniformly on $[0,\,T].$ Therefore, the sequence $\{u_n\}$ converges  in $L^2$ and uniformly on $[0,\,T]\times\R^d$ and the limit satisfies (\ref{mild-sol-colored}).  We can prove uniqueness in a similar way.

We now turn to the proof of the exponential bound. Set
$$
A(t):=\sup_{x\in\rd}\E|u_t(x)|^2.
$$
We claim that
there exist constants $c_4, c_5$ such that  for all $t>0, $ we have
$$
A(t)\leq c_4+ c_5(\lambda L_\sigma)^2\int_0^t \frac{A(s)}{(t-s)^{\beta \gamma/\alpha}}\,\d s.
$$
The renewal inequality in Proposition \ref{prop:renewal-upper} with $\rho=(\alpha-\gamma\beta)/\alpha$  then proves the exponential upper bound. To prove this claim, we start with the mild formulation given by \eqref{mild-sol-colored}, then take the second moment to obtain the following
\begin{equation}
\begin{split}
&\E|u_t(x)|^2=|(\mathcal{G}_B u)_t(x)|^2\\
&\ \ \ \ +\lambda^2 \int_0^t\int_{B\times B}G_B({t-s}, x,y)G_B({t-s}, x,z)f(y,z)\E[\sigma(u_s(y))\sigma(u_s(z))]\d y\d z\d s\\
&\ \ \ \ =I_1+I_2.
\end{split}
\end{equation}

Since $u_0$ is bounded, we have $I_1\leq c_4$. Next we use the assumption on $\sigma$ together with H\"older's inequality to see that
\begin{equation}
\begin{split}
\E[\sigma(u_s(y))\sigma(u_s(z))]&\leq L_\sigma^2 \E[u_s(y)u_s(z)]\\
&\leq  L_\sigma^2[\E|u_s(y)|^2]^{1/2} [\E|u_s(z)|^2]^{1/2}\\
&\leq L_\sigma^2\sup_{x\in \rd}\E|u_s(x)|^2.
\end{split}
\end{equation}
Therefore, using Lemma \ref{lemma:covariance-upper-bound} the second term $I_2$ is thus bounded as follows.
$$
I_2\leq c_5 (\lambda L_\sigma)^2\int_0^t\frac{A(s)}{(t-s)^{\beta\gamma/\alpha}}\, \d s.
$$
Combining the above estimates, we obtain the required result in the claim.

\end{proof}
\subsection{ Proof of  Theorem \ref{excitation-colored}.}
The proof is inspired by the methods in \cite{foondun-tian-liu-2015} and \cite{FooNane}. Set $B=B(0,R)$ and $B_\epsilon=B(0,R-\epsilon)$. We will use the following notation $B^2=B \times B$ and $B_\epsilon^2=B_\epsilon\times B_\epsilon$. The starting point of the proof of the lower bound hinges on the following recursive argument.
\begin{equation*}
\begin{aligned}
\E|u_{t+\tilde{t}}&(x)|^2\\
&=|(\sG_B u)_{t+\tilde{t}}(x)|^2+\lambda^2 \int_0^{t+\tilde{t}}\int_{B^2}\\
& G_B(t+\tilde{t}-s_1,\,x,\,z_1)G_B(t+\tilde{t}-s_1,\,x,\,z_1')\E[\sigma(u_{s_1}(z_1))\sigma(u_{s_1}(z_1'))f(z_1,z_1')]\d z_1\d z_1'\d s_1.
\end{aligned}
\end{equation*}
We now use the assumption that $\sigma(x)\geq l_\sigma|x|$ for all $x$ together with a change of variable to reduce the above to
 \begin{equation*}
\begin{aligned}
\E|u_{t+\tilde{t}}&(x)|^2\\
&\geq|(\sG_B u)_{t+\tilde{t}}(x)|^2+\lambda^2 l_\sigma^2\int_0^{t}\int_{B^2}\\
& G_B(t-s_1,\,x,\,z_1)G_B(t-s_1,\,x,\,z_1')\E|u_{s_1+\tilde{t}}(z_1)u_{s_1+\tilde{t}}(z_1')|f(z_1,z_1')\d z_1\d z_1'\d s_1.
\end{aligned}
\end{equation*}

We also have
\begin{equation*}
\begin{aligned}
&\E|u_{s_1+\tilde{t}}(z_1)u_{s_1+\tilde{t}}(z_1')|\geq|(\sG_B u)_{s_1+\tilde{t}}(z_1)(\sG_B u)_{s_1+\tilde{t}}(z_1)|+\lambda^2l_\sigma^2\int_0^{s_1}\int_{B^2} \\
&G_B(s_1-s_2,\,z_1,\,z_2)G_B(s_1-s_2,\,z_1',\,z_2')\E|u_{s_2+\tilde{t}}(z_2)u_{s_2+\tilde{t}}(z_2')|f(z_2, z_2')\d z_2\d z_2'\d s_2.
\end{aligned}
\end{equation*}

We set $z_0=z_0':=x$ and $s_0:=t$ and continue the recursion as above to obtain
\begin{equation}\label{recursion}
\begin{aligned}
\E|u_{t+\tilde{t}}&(x)|^2\\
&\geq|(\sG_B u)_{t+\tilde{t}}(x)|^2\\
&+ \sum_{k=1}^\infty (\lambda l_\sigma)^{2k}\int_0^t\int_{B^2}\int_0^{s_1}\int_{B^2}\cdots \int_0^{s_{k-1}}\int_{B^2} |(\sG_B u)_{s_k+\tilde{t}}(z_k)(\sG_B u)_{s_k+\tilde{t}}(z_k')|\\
&\prod_{i=1}^kG_B(s_{i-1}-s_{i}, z_{i-1},\,z_i)G_B(s_{i-1}-s_{i}, z'_{i-1},\,z'_i)f(z_i, z_i') \d z_i\d z_i'\d s_i.
\end{aligned}
\end{equation}

\begin{proposition}\label{colorednoiselowbd}
 Fix $\epsilon>0$. Then for all $x\in B(0, R-\epsilon)$ and $0\leq t\leq t_0$
\begin{equation*}
\E|u_{t+\tilde{t}}(x)|^2\geq g_t^2+ g_t^2\sum_{k=1}^\infty \left(\lambda^2l_\sigma^2c_1 \right)^k\left(\frac{t}{k}\right)^{k(\alpha-\gamma \beta )/\alpha},
\end{equation*}
where $c_1$ is a positive constant depending on $\alpha$ and $\gamma$ and $\tilde{t}>0$ is a fixed constant.
\end{proposition}

\begin{proof}
We will look at the following term which comes from the recursive relation described above,
\begin{equation*}
\begin{aligned}
\sum_{k=1}^\infty& (\lambda l_\sigma)^{2k}\int_0^t\int_{B^2}\int_0^{s_1}\int_{B^2}\cdots \int_0^{s_{k-1}}\int_{B^2} |(\sG_B u)_{s_k+\tilde{t}}(z_k)(\sG_B u)_{s_k+\tilde{t}}(z_k')|\\
&\prod_{i=1}^kG_B(s_{i-1}-s_{i}, z_{i-1},\,z_i)G_B(s_{i-1}-s_{i}, z'_{i-1},\,z'_i)f(z_i, z_i') \d z_i\d z_i'\d s_i.
\end{aligned}
\end{equation*}
Using the fact from Remark \ref{lower-bound-frac-int-rep} that for $z_k, z_k' \in B_\epsilon$
\begin{equation}
\begin{split}
(\sG_B u)_{s_k+\tilde{t}}(z_k)(\sG_B u)_{s_k+\tilde{t}}(z_k')& \\
&\geq \inf _{x,y\in B_\epsilon}\inf_{0\leq s\leq t} (\sG_B u)_{s+\tilde{t}}(x)(\sG_B u)_{s+\tilde{t}}(y)\\
&=g_t^2,
\end{split}
\end{equation}
we obtain
\begin{equation*}
\begin{aligned}
\E|u_{t+\tilde{t}}(x)|^2\geq& g_t^2+g_t^2\sum_{k=1}^\infty (\lambda l_\sigma)^{2k}\int_{t-t/k}^t\int_{B_\epsilon^2}\int_{s_1-t/k}^{s_1}\int_{B_\epsilon^2}\cdots \int_{s_{k-1}-t/k}^{s_{k-1}}\int_{B_\epsilon^2} \\
&\prod_{i=1}^kG_B(s_{i-1}-s_{i}, z_{i-1},\,z_i)G_B(s_{i-1}-s_{i}, z'_{i-1},\,z'_i)f(z_i, z_i') \d z_i\d z_i'\d s_i.
\end{aligned}
\end{equation*}
We now make a substitution and reduce the temporal region of integration to write
\begin{equation*}
\begin{aligned}
g_t^2\sum_{k=1}^\infty& (\lambda l_\sigma)^{2k}\int_0^{t/k}\int_{B_\epsilon^2}\int_0^{t/k}\int_{B_\epsilon^2}\cdots \int_0^{t/k}\int_{B_\epsilon^2} \\
&\prod_{i=1}^kG_B(s_{i}, z_{i-1},\,z_i)G_B(s_{i}, z'_{i-1},\,z'_i)f(z_i, z_i') \d z_i\d z_i'\d s_i.
\end{aligned}
\end{equation*}

We will further reduce the domain of integration so the function
\begin{equation*}
\prod_{i=1}^kG_B(s_{i}, z_{i-1},\,z_i)G_B(s_{i}, z'_{i-1},\,z'_i)f(z_i, z_i'),
\end{equation*}
has the required lower bound.  For $i=0,\cdots, k$, we set

\begin{equation*}
z_i\in B(x,\,s_1^{\beta/\alpha}/2)\cap B(z_{i-1},\,s_i^{\beta/\alpha})
\end{equation*}
and
\begin{equation*}
z'_i\in B(x,\,s_1^{\beta/\alpha}/2)\cap B(z'_{i-1},\,s_i^{\beta/\alpha}).
\end{equation*}
We therefore have $|z_i-z'_i|\leq s_1^{\beta/\alpha}$, $|z_i-z_{i-1}|\leq s_i^{\beta/\alpha}$ and $|z'_i-z'_{i-1}|\leq s_i^{\beta/\alpha}$. We use the lower bound on the heat kernel  from Lemma \ref{density-lower-bound-killed-fractional}
\begin{equation*}
\begin{aligned}
\prod_{i=1}^k&G_B(s_{i}, z_{i-1},\,z_i)G_B(s_{i}, z'_{i-1},\,z'_i)f(z_i, z_i')\\
&\geq \frac{c^k}{s_1^{k\gamma\beta/\alpha}}\prod_{i=1}^k\frac{1}{s_i^{2\beta d/\alpha}},
\end{aligned}
\end{equation*}
for some $c>0$.
We set $\sA_i:=B(x,\,s_1^{\beta/\alpha}/2)\cap B(z_{i-1},\,s_i^{\beta/\alpha})$ and $\sA_i':=B(x,\,s_1^{\beta/\alpha}/2)\cap B(z'_{i-1},\,s_i^{\beta/\alpha})$.
We will further choose that $s_i^{\beta/\alpha}\leq \frac{s_1^{\beta/\alpha}}{2}$ and note that $|\sA_i|\geq c_1s_i^{d\beta/\alpha}$ and $|\sA'_i|\geq c_1s_i^{d\beta/\alpha}$. We therefore have

\begin{equation*}
\begin{aligned}
g_t^2\sum_{k=1}^\infty& (\lambda l_\sigma)^{2k}\int_0^{t/k}\int_{B_\epsilon^2}\int_0^{t/k}\int_{B_\epsilon^2}\cdots \int_0^{t/k}\int_{B_\epsilon^2} \\
 &\prod_{i=1}^kG_B(s_{i}, z_{i-1},\,z_i)G_B(s_{i}, z'_{i-1},\,z'_i)f(z_i, z_i') \d z_i\d z_i'\d s_i\\
 &\geq g_t^2\sum_{k=1}^\infty (\lambda l_\sigma)^{2k}\int_0^{t/k}\int_{\sA_1\times\sA'_1}\int_0^{t/k}\int_{\sA_2\times\sA'_2}\cdots \int_0^{t/k}\int_{\sA_k\times\sA'_k}\\
&\frac{1}{s_1^{k\gamma\beta/\alpha}}\prod_{i=1}^k\frac{1}{s_i^{2\beta d/\alpha}}\d z_i\d z_i'\d s_i\\
&\geq g_t^2\sum_{k=1}^\infty(\lambda l_\sigma c_2)^{2k} \int_0^{t/k}\frac{1}{s_1^{k\gamma \beta/\alpha}}s_1^{k-1}\d s_1\\
&\geq g_t^2\sum_{k=1}^\infty(\lambda l_\sigma c_3)^{2k} \left(\frac{t}{k}\right)^{k(1-\gamma\beta/\alpha)}.
\end{aligned}
\end{equation*}

This completes the proof of proposition.
\end{proof}

\begin{proposition}
For any fixed $\epsilon>0$, there exists a $t_0>0$ such that for all $0<t\leq t_0$,
$$
\liminf_{\lambda\to \infty} \frac{\log \log \mathcal{I}_{\epsilon, t}}{\log \lambda}\geq \frac{2\alpha}{\alpha-\beta\gamma}.
$$
\end{proposition}

\begin{proof}
Since $t$ is strictly positive, we can use a substitution and rewrite the lower bound in Proposition \ref{colorednoiselowbd} as
\begin{equation}
\begin{split}
\sum_{k=1}^\infty(\lambda l_\sigma c_3)^{2k} \left(\frac{t}{k}\right)^{k(1-\gamma\beta/\alpha)}&\\
&=\sum_{k=1}^\infty \left( \frac{(\lambda l_\sigma c_3)^2 t^{(\alpha-\gamma\beta)/\alpha}}{k^{(\alpha-\gamma\beta)/\alpha}}\right)^k.
\end{split}
\end{equation}

Now from Lemma \ref{lower-bound-lambda}  with $\rho=(\alpha-\gamma\beta)/\alpha$ and $\theta=\lambda^2$ together with the previous proposition give the result.

\end{proof}

\begin{proof}[Proof of Theorem \ref{excitation-colored}]

The previous propositions prove the theorem for all $0<t\leq t_0$. Now we extend the result to all $t>0$. For any $T,t>0$,
\begin{equation*}
\begin{aligned}
\E|u_{t+T}&(x)|^2\\
&\geq|(\sG_B u)_{t+T}(x)|^2+\lambda^2 l_\sigma^2\int_0^{t+T}\int_{B^2}\\
& G_B(T+t-s_1,\,x,\,z_1)G_B(T+t-s_1,\,x,\,z_1')\E|u_{s_1}(z_1)u_{s_1}(z_1')|f(z_1,z_1')\d z_1\d z_1'\d s_1.
\end{aligned}
\end{equation*}

This leads to
\begin{equation*}
\begin{aligned}
\E|u_{t+T}&(x)|^2\\
&\geq|(\sG_B u)_{t+T}(x)|^2+\lambda^2 l_\sigma^2\int_0^{t}\int_{B^2}\\
& G_B(t-s_1,\,x,\,z_1)G_B(t-s_1,\,x,\,z_1')\E|u_{T+s_1}(z_1)u_{T+s_1}(z_1')|f(z_1,z_1')\d z_1\d z_1'\d s_1.
\end{aligned}
\end{equation*}
A similar argument used in the proof of Proposition \ref{colorednoiselowbd} shows that
\begin{equation*}
\begin{aligned}
\E|u_{t+T}&(x)|^2\\
\geq &(\sG_B u)_{t+T}(x)|^2\\
+&\sum_{k=1}^\infty (\lambda l_\sigma)^{2k}\int_0^t\int_{B^2}\int_0^{s_1}\int_{B^2}\cdots \int_0^{s_{k-1}}\int_{B^2} |(\sG_B u)_{T+s_k}(z_k)(\sG_B u)_{T+s_k}(z_k')|\\
&\prod_{i=1}^kG_B(s_{i-1}-s_{i}, z_{i-1},\,z_i)G_B(s_{i-1}-s_{i}, z'_{i-1},\,z'_i)f(z_i, z_i') \d z_i\d z_i'\d s_i.\end{aligned}
\end{equation*}
Similar ideas to those used in the proof of Proposition \ref{colorednoiselowbd} combined with  the proof of the proposition above show that for all $t\leq t_0$, we have

$$
\liminf_{\lambda\to \infty} \frac{\log \log \E|u_{T+t}(x)|^2}{\log \lambda}\geq \frac{2\alpha}{\alpha-\beta\gamma},
$$
for all $T>0$ and whenever  $x\in B(0, R-\epsilon)$.
\end{proof}
\begin{proof}[Proof of Corollary \ref{corollary-excitation-colored}.]
The proof of this corollary  is exactly as that of Corollary \ref{corollary-excitation} and it is omitted.
\end{proof}

\section{An  extension}
We can obtain results similar to the results in section 5 of \cite{foondun-tian-liu-2015}. We state one example, other examples can also be extended to  the time fractional case. We choose $\mathcal{L}$ to be the generator of the relativistic  stable process killed upon exiting the ball $B(0,R)$.
So we are looking at the following equation
\begin{equation}\label{tfspde-colored-relativistic}
\begin{split}
\partial^\beta_t u_t(x)&=m u_t(x)-(m^{2/\alpha}-\Delta)^{\alpha/2}u_t(x)+I_t^{1-\beta}[\lambda \sigma(u_t(x))\dot{F}(t,\,x)],\, x\in B(0,R),\\
u_t(x)&=0,\ \ \ x\in B(0,R)^c.
\end{split}
\end{equation}
Here $m$ is a positive number. It is known that for any $\epsilon>0$, there exists a $T_0>0$ such that for all $x,y\in B(0, R-\epsilon)$ and $t\leq T_0$ we have
$$
p(t,x,y)\approx  t^{-d/\alpha}
$$
whenever  $|x-y|\leq t^{1/\alpha}$. See, for instance, \cite{chen-kim-song-2012}. The constant involved in this inequality depends on $m$.  We therefore have the same conclusion as the one in Theorem  \ref{excitation-colored}. So we have for all $x\in B(0,R-\epsilon)$
\begin{equation*}
\lim_{\lambda\rightarrow \infty} \frac{\log \log \E|u_t(x)|^2}{\log \lambda}=\frac{2\alpha}{\alpha-\gamma\beta},
\end{equation*}
where $u_t$ is the unique solution to \eqref{tfspde-colored-relativistic}.

\end{document}